\documentclass[12pt]{article}
\usepackage{amsmath,amsthm,amssymb,amsfonts,bbm,hyperref}

\title{Exponential growth of ponds in invasion percolation on regular trees}
\author{Jesse Goodman}

\newcommand{\E}{\mathbb{E}}

\newcommand{\N}{\mathbb{N}}
\newcommand{\Z}{\mathbb{Z}}
\renewcommand{\P}{\mathbb{P}}

\newcommand{\condP}[2]{\mathbb{P}\left(\left.#1\,\right\vert#2\right)}
\newcommand{\condE}[2]{\mathbb{E}\left(\left.#1\,\right\vert#2\right)}

\newcommand{\indicator}[1]{\mathbbm{1}_{\left(#1\right)}}

\newcommand{\given}{\,\big\vert\,}
\newcommand{\boundary}{\partial}
\newcommand{\goesto}{\rightarrow}
\newcommand{\goesweaklyto}{\Rightarrow}
\newcommand{\connectsto}{\leftrightarrow}

\newcommand{\set}[1]{\left\{#1\right\}}
\newcommand{\abs}[1]{\left\vert#1\right\vert}
\newcommand{\floor}[1]{\left\lfloor#1\right\rfloor}
\newcommand{\ceil}[1]{\left\lceil#1\right\rceil}

\renewcommand{\th}{\text{th}}

\renewcommand{\phi}{\varphi}

\renewcommand{\d}[1]{\,d#1\,}
\newcommand{\dx}{\d{x}}
\newcommand{\dy}{\d{y}}

\newcounter{constantsubscript}[subsection]
\newcommand{\newnumberedconstant}{\addtocounter{constantsubscript}{1}c_{\arabic{constantsubscript}}}
\newcommand{\newnumberedConstant}{\addtocounter{constantsubscript}{1}C_{\arabic{constantsubscript}}}
\newcommand{\namednumberedconstant}[1]{\newnumberedconstant\newcounter{@constantsubscript@#1}\setcounter{@constantsubscript@#1}{\value{constantsubscript}}}
\newcommand{\namednumberedConstant}[1]{\newnumberedConstant\newcounter{@Constantsubscript@#1}\setcounter{@Constantsubscript@#1}{\value{constantsubscript}}}

\newcommand{\lastConstant}{C_{\arabic{constantsubscript}}}
\newcommand{\previousconstant}[1]{c_{\arabic{@constantsubscript@#1}}}
\newcommand{\previousConstant}[1]{C_{\arabic{@Constantsubscript@#1}}}

\newcommand{\calC}{{\cal{C}}}
\newcommand{\underlyinggraph}{{\cal{G}}}

\newtheorem{thm}{Theorem}[section]
\newtheorem{lemma}[thm]{Lemma}
\newtheorem{prop}[thm]{Proposition}

\newtheorem{coro}[thm]{Corollary}

\theoremstyle{definition}

\numberwithin{equation}{section}

\begin{document}

\maketitle

\abstract{In invasion percolation, the edges of successively maximal weight (the \emph{outlets}) divide the invasion cluster into a chain of \emph{ponds} separated by outlets.  On the regular tree, the ponds are shown to grow exponentially, with law of large numbers, central limit theorem and large deviation results.  The tail asymptotics for a fixed pond are also studied and are shown to be related to the asymptotics of a critical percolation cluster, with a logarithmic correction.}

\section{Introduction and definitions}

\subsection{The model: invasion percolation, ponds and outlets}

Consider an infinite connected locally finite graph $\underlyinggraph$, with a distinguished vertex $o$, the root.  On each edge, place an independent Uniform$[0,1]$ edge weight, which we may assume (a.s.) to be all distinct.  Starting from the subgraph $\calC_0=\set{o}$, inductively grow a sequence of subgraphs $\calC_i$ according to the following deterministic rule.  At step $i$, examine the edges on the boundary of $C_{i-1}$, and form $C_i$ by adjoining to $\calC_{i-1}$ the edge whose weight is minimal.  The infinite union 
\begin{equation}
\calC=\bigcup_{i=1}^\infty \calC_i
\end{equation}
is called the \emph{invasion cluster}.

Invasion percolation is closely related to ordinary (Bernoulli) percolation.  For instance, (\cite{CCN1985} for $\underlyinggraph = Z^d$; later greatly generalized by \cite{HPS1999}) if $\underlyinggraph$ is quasi-transitive, then for any $p>p_c$, only a finite number of edges of weight greater than $p$ are ever invaded.  On the other hand, it is elementary to show that for any $p<p_c$, infinitely many edges of weight greater than $p$ must be invaded.  In other words, writing $\xi_i$ for the weight of the $i^\th$ invaded edge, we have 
\begin{equation}
\label{limsupispcGeneral}
\limsup_{i\goesto\infty} \xi_i = p_c
\end{equation}
So invasion percolation produces an infinite cluster using only slightly more than critical edges, even though there may be no infinite cluster at criticality.  The fact that invasion percolation is linked to the critical value $p_c$, even though it contains no parameter in its definition, makes it an example of \emph{self-organized criticality}.

Under mild hypotheses (see section \ref{GeneralMarkovStructureSubsection}), the invasion cluster has a natural decomposition into \emph{ponds} and \emph{outlets}.  Let $e_1\in\calC$ be the edge whose weight $Q_1$ is the largest ever invaded.  For $n>1$, $e_n$ is the edge in $\calC$ whose weight $Q_n$ is the highest among edges invaded after $e_{n-1}$.  We call $e_n$ the \emph{$n^\th$ outlet} and $Q_n$ the corresponding \emph{outlet weight}.  Write $\hat{V}_n$ for the step at which $e_n$ was invaded, with $\hat{V}_0=0$.  The \emph{$n^\th$ pond} is the subgraph of edges invaded at steps $i\in (\hat{V}_{n-1}, \hat{V}_n]$.

Suppose an edge $e$, with weight $p$, is first examined at step $i\in(\hat{V}_{n-1}, \hat{V}_n]$.  (That is, $i$ is the first step at which $e$ is on the boundary of $\calC_{i-1}$.)  Then we have the following dichotomy: either 
\begin{itemize}
\item
$e$ will be invaded as part of the $n^\th$ pond (if $p\leq Q_n$); or 
\item
$e$ will never be invaded (if $p>Q_n$)
\end{itemize}
This implies that the ponds are connected subgraphs and touch each other only at the outlets.  Moreover, the outlets are pivotal in the sense that any infinite non-intersecting path in $\calC$ starting at $o$ must pass through every outlet.  Consequently $\calC$ is decomposed as an infinite chain of ponds, connected at the outlets.

In this paper we take $\underlyinggraph$ to be a regular tree and analyze the asymptotic behaviour of the ponds, the outlets and the outlet weights.  This problem can be approached in two directions: by considering the ponds as a sequence and studying the growth properties of that sequence; or by considering a fixed pond and finding its asymptotics.  We will see that the sequence of ponds grows exponentially, with exact exponential constants.  For a fixed pond, its asymptotics correspond to those of ordinary percolation with a logarithmic correction.  

These computations are based on representing $C$ in terms of the outlet weights $Q_n$, as in \cite{AGdHS2008}.  Conditional on $(Q_n)_{n=0}^\infty$, each pond is an independent percolation cluster with parameter related to $Q_n$.  In particular, the fluctuations of the ponds are a combination of fluctuations in $Q_n$ and the additional randomness.

Surprisingly, in all but the large deviation sense, the asymptotic behaviour for the ponds is controlled by the outlet weights alone: the remaining randomness after conditioning only on $(Q_n)_{n=0}^\infty$ disappears in the limit, and the fluctuations are attributable solely to fluctuations of $Q_n$.

\subsection{Known results}

The terminology of ponds and outlets comes from the following description (see \cite{vdBJV2007}) of invasion percolation.  Consider a random landscape where the edge weights represent the heights of channels between locations.  Pour water into the landscape at $o$; then as more and more water is added, it will flow into neighbouring areas according to the invasion percolation mechanism.  The water level at $o$, and throughout the first pond, will rise until it reaches the height of the first outlet.  Once water flows over an outlet, however, it will flow into a new pond where the water will only ever rise to a lower height.  Note that the water level in the $n^\th$ pond is the height (edge weight) of the $n^\th$ outlet.

The edge weights may also be interpreted as energy barriers for a random walker exploring a random energy landscape: see \cite{NewmanStein1995}.  If the energy levels are highly separated, then (with high probability and until some large time horizon) the walker will visit the ponds in order, spending a long time in each pond before crossing the next outlet.  In this interpretation the growth rate of the ponds determines the effect of entropy on this analysis.  See the extended discussion in \cite{NewmanStein1995}.

Invasion percolation is also related to the incipient infinite cluster (IIC), at least in the cases $\underlyinggraph=\Z^2$ (\cite{Jarai2003}) and $\underlyinggraph$ a regular tree: see, e.g., \cite{Jarai2003}, \cite{AGdHS2008} and \cite{DS2009}.  For a cylinder event $E$, the law of the IIC can be defined by
\begin{equation}
\label{IICdefinition}
\P_{\text{IIC}}(E)\overset{def}{=}\lim_{k\goesto\infty}\P_{p_c}(E\given o\connectsto\boundary B(k))
\end{equation}
or by other limiting procedures, many of which can be proved to be equivalent to each other.  Both the invasion cluster and the IIC consist of an infinite cluster that is ``almost critical'', in view of \eqref{limsupispcGeneral} or \eqref{IICdefinition} respectively.  For $\underlyinggraph=\Z^2$ (\cite{Jarai2003}) and $\underlyinggraph$ a regular tree (\cite{AGdHS2008}), the IIC can be defined in terms of the invasion cluster: if $X_k$ denotes a vertex chosen uniformly from among the invaded vertices within distance $k$ of $o$, and $\tau_{X_k}E$ denotes the translation of $E$ when $o$ is sent to $X_k$, then
\begin{equation}
\P_{\text{IIC}}(E)=\lim_{k\goesto\infty} P(\tau_{{X_k}}E)
\end{equation}
Surprisingly, despite this local equivalence, the invasion cluster and the IIC are globally different: they are mutually singular and, at least on the regular tree, have different scaling limits, although they have the same scaling exponents.

The regular tree case, first considered in \cite{NickWilk1983}, was studied in great detail in \cite{AGdHS2008}.  Any infinite non-intersecting path from $o$ must pass through every outlet; on a tree, this implies that there is a \emph{backbone}, the unique infinite non-intersecting path from $o$.  In \cite{AGdHS2008} a description of the invasion cluster was given in terms of the \emph{forward maximal weight process}, the outlet weights indexed by height along the backbone (see section \ref{InvasionClusterStructureSubsection}).  This parametrization in terms of the external geometry of the tree allowed the calculation of natural geometric quantities, such as the number of invaded edges within a ball.  In the following, we will see that when information about the heights is discarded, the process of edge weights takes an even simpler form.

The detailed structural information in \cite{AGdHS2008} was used in \cite{AngelGMerle} to identify the scaling limit of the invasion cluster (again for the regular tree).  Since the invasion cluster is a tree with a single infinite end, it can be encoded by its Lukaciewicz path or its height and contour functions.  Within each pond, the scaling limit of the Lukaciewicz path is computed, and the different ponds are stitched together to provide the full scaling limit.

The two-dimensional case was also studied in a series of papers by van den Berg, Damron, J\'{a}rai, Sapozhnikov and V\'{a}gv\"{o}lgyi (\cite{vdBJV2007}, \cite{DSV2008} and \cite{DS2009}).  There they study, among other things, the probability that the $n^\th$ pond extends a distance $k$ from $o$, for $n$ fixed.  For $n=1$ this is asymptotically of the same order as the probability that a critical percolation cluster extends a distance $k$, and for $n>1$ there is a correction factor $(\log k)^{n-1}$.  Furthermore an exponential growth bound for the ponds is given.  This present work was motivated in part by the question of what the corresponding results would be for the tree.  Quite remarkably, they are essentially the same, suggesting that a more general phenomenon may be involved.

In the results and proofs that follow, we shall see that the sequence of outlet weights plays a dominant role.  Indeed, all of the results in Theorems \ref{SLLNtheorem}--\ref{QnLnAsymp} are proved first for $Q_n$, then extended to other pond quantities using conditional tail estimates.  Consequently, all of the results can be understood as consequences of the growth mechanism for the sequence $Q_n$.  On the regular tree, we are able to give an exact description of the sequence $Q_n$ in terms of a sum of independent random variables (see section \ref{RegularTreeCaseSubsection}).  In more general graphs, this representation cannot be expected to hold exactly.  However, the similarities between the between the pond behaviours, even on graphs as different as the tree and $\Z^2$, suggest that an approximate analogue may hold.  Such a result would provide a unifying explanation for both the exponential pond growth and the asymptotics of a fixed pond, even on potentially quite general graphs.

\subsection{Summary of notation}

We will primarily consider the case where $\underlyinggraph$ is the forward regular tree of degree $\sigma$: namely, the tree in which the root $o$ has degree $\sigma$ and every other vertex has degree $\sigma+1$.  The weight of the $i^\th$ invaded edge is $\xi_i$.  The $n^\th$ outlet is $e_n$ and its edge weight is $Q_n$.  We may naturally consider $e_n$ to be an oriented edge $e_n=(\underline{v}_n,\overline{v}_n)$, where $\underline{v}_n$ is invaded before $\overline{v}_n$.  The step at which $e_n$ is invaded is denoted $\hat{V}_n$ and the (graph) distance from $o$ to $\overline{v}_n$ is $\hat{L}_n$.  Setting $\hat{V}_0=\hat{L}_0=0$ for convenience, we write $V_n=\hat{V}_n-\hat{V}_{n-1}$ and $L_n=\hat{L}_n-\hat{L}_{n-1}$.  

There is a natural geometric interpretation of $L_n$ as the length of the part of the backbone in the $n^\th$ pond, and $V_n$ as the volume (number of edges) of the $n^\th$ pond.  In particular $\hat{V}_n$ is the volume of the union of the first $n$ ponds.

$R_n$ is the length of the longest upward-pointing path in the $n^\th$ pond, and $R'_n$ is the length of the longest upward-pointing path in the union of the first $n$ ponds.

We shall later work with the quantity $\delta_n$; for its definition, see \eqref{deltanDefinition}.

We note the following elementary relations:
\begin{gather}
\hat{L}_n=\sum_{i=1}^n L_i,\qquad
\hat{V}_n=\sum_{i=1}^n V_i,\\
Q_{n+1}<Q_n,\qquad
L_n\leq R_n\leq R'_n\leq \sum_{i=1}^n R_i.
\end{gather}

Probability laws will generically be denoted $\P$.  For $p\in[0,1]$, $\P_p$ denotes the law of Bernoulli percolation with parameter $p$.  For a set $A$ of vertices, the event $\set{x\connectsto A}$ means that there is a path of open edges joining $x$ to some point of $A$, and $\set{x\connectsto\infty}$ means that there is an infinite non-intersecting path of open edges starting at $x$.  We define the percolation probability $\theta(p)=\P_p(o\connectsto\infty)$ and $p_c=\inf\set{p: \theta(p)>0}$.  $\boundary B(k)$ denotes the vertices at distance exactly $k$ from $o$.

For non-zero functions $f(x)$ and $g(x)$, we write $f(x)\sim g(x)$ if $\lim \frac{f(x)}{g(x)}=1$; the point at which the limit is to be taken will usually be clear from the context.  We write $f(x)\asymp g(x)$ if there are constants $c$ and $C$ such that $cg(x)\leq f(x)\leq Cg(x)$.

\section{Main results}

\subsection{Exponential growth of the ponds}

Let $\vec{Z}_n$ denote the 7-tuple
\begin{align}
\begin{split}
&
\vec{Z}_n=\Bigl(\log\!\left((Q_n-p_c)^{-1}\right)\! , \log L_n, \log \hat{L}_n, \\
&\qquad\qquad
\log R_n, \log R'_n, \tfrac{1}{2}\log V_n, \tfrac{1}{2}\log \hat{V}_n \Bigr)
\end{split}
\end{align}
and write $\mathbbm{1}=(1,1,1,1,1,1,1)$.

\begin{thm}
\label{SLLNtheorem}
With probability $1$, 
\begin{equation}
\label{ZnSLLN}
\lim_{n\goesto\infty} \frac{\vec{Z}_n}{n}=\mathbbm{1}.
\end{equation}
\end{thm}

\begin{thm}
\label{CLTtheorem}
If $(B_t)_{t\geq 0}$ denotes a standard Brownian motion then 
\begin{equation}
\label{ZnCLT}
\left(\frac{\vec{Z}_{\floor{Nt}}-Nt\cdot\mathbbm{1}}{\sqrt{N}}\right)_{t\geq 0} \goesweaklyto (B_t\cdot\mathbbm{1})_{t\geq 0}
\end{equation}
as $N\goesto\infty$, with respect to the metric of uniform convergence on compact intervals of $t$.
\end{thm}

These theorems say that each component of $\vec{Z}$ satisfies a law of large numbers and functional central limit theorem, with the same limiting Brownian motion for each component.

Theorem \ref{CLTtheorem} shows that the logarithmic scaling in Theorem \ref{SLLNtheorem} cannot be replaced by a linear rescaling such as $e^n(Q_n-p_c)$.  Indeed, $\log((Q_n-p_c)^{-1})$ has characteristic additive fluctuations of order $\pm\sqrt{n}$, and therefore $Q_n-p_c$ fluctuates by a multiplicative factor of the form $e^{\pm\sqrt{n}}$.  As $n\goesto\infty$ this will be concentrated at $0$ and $\infty$, causing tightness to fail.

\begin{thm}
\label{LDtheorem}
$\frac{1}{n}\log\!\left((Q_n-p_c)^{-1}\right)$ satisfies a large deviation principle on $[0,\infty)$ with rate $n$ and rate function
\begin{equation}
\label{RateFunctionphi}
\phi(u)=u-\log u -1.
\end{equation}
$\frac{1}{n}\log L_n$, $\frac{1}{n}\log R_n$ and $\frac{1}{2n}\log V_n$ satisfy large deviation principles on $[0,\infty)$ with rate $n$ and rate function $\psi$, where
\begin{equation}
\psi(u)=
\begin{cases}
u-\log u -1&\text{if $u\geq \frac{1}{2}$,}\\
\log (2) - u&\text{if $u\leq \frac{1}{2}$.}
\end{cases}
\end{equation}
\end{thm}
It will be shown that $\psi$ arises as the solution of the variational problem
\begin{equation}
\psi(u)=
\inf_{v\geq u} \bigl(\phi(v)+v-u\bigr)
\end{equation}

\subsection{Tail behaviour of a pond}

Theorems \ref{SLLNtheorem}--\ref{LDtheorem} describe the growth of the ponds as a sequence.  We now consider a fixed pond and study its tail behaviour.

\begin{thm}
\label{QnLnAsymp}
For $n$ fixed and $\epsilon\goesto 0^+$, $k\goesto\infty$,
\begin{align}
\label{QnTail}
\P\left(Q_n<p_c(1+\epsilon)\right)
&\sim
\frac{2\sigma}{\sigma-1}\frac{\epsilon \left(\log \epsilon^{-1}\right)^{n-1}}{(n-1)!}
\end{align}
and
\begin{align}
\label{LnTail}
\P\left(L_n>k\right) \sim \P\left(\hat{L}_n > k\right)
&\sim
\frac{2\sigma}{\sigma-1}\frac{(\log k)^{n-1}}{k(n-1)!}\\
\label{RnTail}
\P\left(R_n>k\right) \asymp \P\left(R'_n > k\right)
&\asymp
\frac{(\log k)^{n-1}}{k}\\
\label{VnTail}
\P\left(V_n>k\right) \asymp \P\left(\hat{V}_n > k\right)
&\asymp
\frac{(\log k)^{n-1}}{\sqrt{k}}
\end{align}
\end{thm}

Using the well-known asymptotics 
\begin{align}
\theta(p)
&\sim 
\frac{2\sigma^2}{\sigma-1}(p-p_c)
&&\text{as $p\goesto p_c^+$,}\\
\P_{p_c}(o\connectsto\boundary B(k))
&\sim 
\frac{2\sigma}{(\sigma-1)k} 
&&\text{as $k\goesto\infty$,}
\end{align}
we may rewrite \eqref{QnTail}--\eqref{VnTail} as
\begin{align}
\P\left(Q_n<p_c(1+\epsilon)\right)
&\sim
\frac{\left(\log \epsilon^{-1}\right)^{n-1}}{(n-1)!} \theta(p_c(1+\epsilon))\\
\P\left(L_n>k\right) \sim \P\left(\hat{L}_n > k\right) 
&\sim
\frac{(\log k)^{n-1}}{(n-1)!}\P_{p_c}(o\connectsto\boundary B(k))\\
\P\left(R_n>k\right) \asymp \P\left(R'_n > k\right)
&\asymp
(\log k)^{n-1}\P_{p_c}(o\connectsto\boundary B(k))\\
\P\left(V_n>k\right) \asymp \P\left(\hat{V}_n > k\right)
&\asymp
(\log k)^{n-1}\P_{p_c}(\abs{C(o)}>k)
\end{align}

Working in the case $\underlyinggraph=\Z^2$, \cite{DSV2008} considers $\tilde{R}_n$, the maximum distance from $o$ to a point in the first $n$ ponds, which is essentially $R'_n$ in our notation.  \cite[Theorem 1.5]{DSV2008} states that 
\begin{equation}
\label{Z2Rasymp}
\P(\tilde{R}_n\geq k)\asymp(\log k)^{n-1}\P_{p_c}(o\connectsto\boundary B(k))
\end{equation}
and notes as a corollary
\begin{equation}
\label{Z2AsymptoPercWithDefects}
\P(\tilde{R}_n\geq k)\asymp\P_{p_c}(o\overset{n-1}{\longleftrightarrow} \boundary B(k))
\end{equation}
where $\overset{i}{\connectsto}$ denotes a percolation connection where up to $i$ edges are allowed to be vacant (``percolation with defects'').  \eqref{Z2AsymptoPercWithDefects} suggests the somewhat plausible heuristic of approximating the union of the first $n$ ponds by the set of vertices reachable by critical percolation with at most $n-1$ defects.  Indeed, the proof of \eqref{Z2Rasymp} uses in part a comparison to percolation with defects.  By contrast, on the tree the following result holds:

\begin{thm}
\label{TreePercWithDefectsTheorem}
For fixed $n$ and $k\goesto\infty$,
\begin{equation}
\label{TreePercWithDefectsAsymp}
\P_{p_c}(o\overset{n}{\connectsto}\boundary B(k))\asymp k^{-2^{-n}}
\end{equation}
\end{thm}

The dramatic contrast between \eqref{Z2AsymptoPercWithDefects} and \eqref{TreePercWithDefectsAsymp} can be explained in terms of the number of large clusters in a box.  In $\Z^2$, a box of side length $S$ has generically only one cluster of diameter of order $S$.  On the tree, by contrast, there are many large clusters.  Indeed, a cluster of size $N$ has of order $N$ edges on its outer boundary, any one of which might be adjacent to another large cluster, independently of every other edge.  Percolation with defects allows the best boundary edge to be chosen, whereas invasion percolation is unlikely to invade the optimal edge.

\subsection{Outline of the paper}

Section \ref{GeneralMarkovStructureSubsection} states a Markov property for the outlet weights that is valid for any graph.  From section \ref{InvasionClusterStructureSubsection} onwards, we specialize to the case where $\underlyinggraph$ is a regular tree.  In section \ref{InvasionClusterStructureSubsection} we recall results from \cite{AGdHS2008} that describe the structure of the invasion cluster conditional on the outlet weights $Q_n$.  Section \ref{RegularTreeCaseSubsection} analyzes the Markov transition mechanism of section \ref{GeneralMarkovStructureSubsection} and proves the results of Theorems \ref{SLLNtheorem}--\ref{LDtheorem} for $Q_n$.

Section \ref{PondTailBoundsSection} states conditional tail bounds for $L_n$, $R_n$ and $V_n$ given $Q_n$.  The rest of sections \ref{LLNandCLTproofSection}--\ref{TailAsymptoticsSection} use these tail bounds to prove Theorems \ref{SLLNtheorem}--\ref{QnLnAsymp}.  The proof of the bounds in section \ref{PondTailBoundsSection} is given in section \ref{PondTailBoundsProofSection}.  Finally, section \ref{PercolationWithDefectsSection} gives the proof of Theorem \ref{TreePercWithDefectsTheorem}.

\section{Markov structure of invasion percolation}

In section \ref{GeneralMarkovStructureSubsection} we give sufficient conditions for the existence of ponds and outlets, and state a Markov property for the ponds, outlets and outlet weights.  Section \ref{InvasionClusterStructureSubsection} summarizes some previous results from \cite{AGdHS2008} concerning the structure of the invasion cluster.  Finally in section \ref{RegularTreeCaseSubsection} we analyze the resulting Markov chain in the special case where $\underlyinggraph$ is a regular tree and prove the results of Theorems \ref{SLLNtheorem}--\ref{LDtheorem} for $Q_n$.

\subsection{General graphs: ponds, outlets and outlet weights}
\label{GeneralMarkovStructureSubsection}

The representation of an invasion cluster in terms of ponds and outlets is guaranteed to be valid under the following two assumptions:
\begin{equation}
\label{NoCriticalPercolation}
\theta(p_c)=0
\end{equation}
and
\begin{equation}
\label{limsupispc}
\limsup_{i\goesto\infty} \xi_i=p_c\quad\text{a.s.}
\end{equation}
\eqref{NoCriticalPercolation} is known to hold for many graphs and is conjectured to hold for any transitive graph for which $p_c<1$ (\cite[Conjecture 4]{BenjaminiSchramm1996}; see also, for instance, \cite[section 8.3]{LyonsPeresProbOnTreesNets}).  If the graph $\underlyinggraph$ is quasi-transitive, \eqref{limsupispc} follows from the general result \cite[Proposition 3.1]{HPS1999}.  Both \eqref{NoCriticalPercolation} and \eqref{limsupispc} hold when $\underlyinggraph$ is a regular tree.

The assumption \eqref{NoCriticalPercolation} implies that w.p. 1,
\begin{equation}
\label{xiAbovepcIO}
\sup_{i> i_0} \xi_i > p_c\quad\text{for all $i_0$}
\end{equation}
since otherwise there would exist somewhere an infinite percolation cluster at level $p_c$.  We can then make the inductive definition
\begin{align}
Q_1&=\max_{i>0} \xi_i = \xi_{\hat{V}_1}\\
Q_n&=\max_{i>\hat{V}_{n-1}} \xi_i = \xi_{\hat{V}_n}\quad(n>1)
\end{align}
since \eqref{limsupispc} and \eqref{xiAbovepcIO} imply that the maxima are attained.

Condition on $Q_n$, $e_n$, and the union $\tilde{\calC}_n$ of the first $n$ ponds.  We may naturally consider $e_n$ to be an oriented edge $e_n=(\underline{v}_n,\overline{v}_n)$ where the vertex $\underline{v}_n$ was invaded before $\overline{v}_n$.  The condition that $e_n$ is an outlet, with weight $Q_n$, implies that there must exist an infinite path of edges with weights at most $Q_n$, starting from $\overline{v}_n$ and remaining in $\underlyinggraph\backslash \tilde{\calC}_n$.  However, the law of the edge weights in $\underlyinggraph\backslash \tilde{\calC}_n$ is not otherwise affected by $Q_n,e_n,\tilde{\calC}_n$.  In particular we have
\begin{equation}
\label{RecursiveLawofQn}
\condP{Q_{n+1}<q'}{Q_n,e_n,\tilde{\calC}_n}
=\frac{\P_{q'}(\overline{e}_n\connectsto\infty\text{ in $\underlyinggraph\backslash \tilde{\calC}_n$})}{\P_{Q_n}(\overline{e}_n\connectsto\infty\text{ in $\underlyinggraph\backslash \tilde{\calC}_n$})}
\end{equation}
on the event $\set{q'\leq Q_n}$.  In \eqref{RecursiveLawofQn} we can replace $\underlyinggraph\backslash \tilde{\calC}_n$ by the connected component of $\underlyinggraph\backslash \tilde{\calC}_n$ that contains $\overline{e}_n$.

\subsection{Geometric structure of the invasion cluster: the regular tree case}
\label{InvasionClusterStructureSubsection}

In \cite[section 3.1]{AGdHS2008}, the same outlet weights are studied, parametrized by height rather than by pond.  $W_k$ is defined to be the maximum invaded edge weight above the vertex at height $k$ along the backbone.  

A key point in the analysis in \cite{AGdHS2008} is the observation that $(W_k)_{k=0}^\infty$ is itself a Markov process.  $W_k$ is constant for long stretches, corresponding to $k$ in the same pond, and the jumps of $W_k$ occur when an outlet is encountered.  The relation between the two processes is given by
\begin{equation}
\label{WkQnRelation}
W_{k}=Q_n\quad\text{iff}\quad \hat{L}_{n-1} \leq k < \hat{L}_n
\end{equation}

From \eqref{WkQnRelation} we see that the $(Q_n)_{n=0}^\infty$ are the successive distinct values of $(W_k)_{k=0}^\infty$, and $L_n=\hat{L}_n-\hat{L}_{n-1}$ is the length of time the Markov chain $W_k$ spends in state $Q_n$ before jumping to state $Q_{n+1}$.  In particular, $L_n$ is geometric conditional on $Q_n$, with some parameter depending only on $Q_n$.  As we will refer to it often, we define $\delta_n$ to be that geometric parameter:
\begin{equation}
\label{deltanDefinition}
\condP{L_n>m\,}{Q_n}=(1-\delta_n)^m
\end{equation}
A further analysis (see \cite[section 2.1]{AGdHS2008}) shows that the off-backbone part of the $n^\th$ pond is a sub-critical Bernoulli percolation cluster with a parameter depending on $Q_n$, independently in each pond.  We summarize these results in the following theorem.

\begin{thm}[\cite{AGdHS2008}, sections 2.1 and 3.1]
\label{InvasionStructureTheorem}
Conditional on $(Q_n)_{n=1}^\infty$, the $n^\th$ pond of the invasion cluster consists of
\begin{enumerate}
\item
$L_n$ edges from the infinite backbone, where $L_n$ is geometric with parameter $\delta_n$; and
\item
emerging along the $\sigma-1$ sibling edges of each backbone edge, independent sub-critical Bernoulli percolation clusters with parameter
\begin{equation}
\label{SubcriticalParameter}
p_c(1-\delta_n)
\end{equation}
\end{enumerate}
Given $(Q_n)_{n=0}^\infty$, the ponds are conditionally independent for different $n$.  $\delta_n$ is a continuous, strictly increasing functions of $Q_n$ and satisfies
\begin{equation}
\label{QndeltanRelation}
\delta_n\sim \frac{\sigma-1}{2\sigma}\theta(Q_n)\sim \sigma (Q_n-p_c)
\end{equation}
as $Q_n\goesto p_c^+$.
\end{thm}
The meaning of \eqref{QndeltanRelation} is that $\delta_n=f(Q_n)$ where $f(q)\sim \frac{\sigma-1}{2\sigma}\theta(q)\sim \sigma(q-p_c)$ as $q\goesto p_c^+$. 

It is not at first apparent that the geometric parameter $\delta_n$ in \eqref{deltanDefinition} is the same quantity that appears in \eqref{SubcriticalParameter}, and indeed \cite{AGdHS2008} has two different notations for the two quantities: see \cite[equations (3.1) and (2.14)]{AGdHS2008}.  Combining equations (2.3), (2.5), (2.14) and (3.1) of \cite{AGdHS2008} shows that they are equivalent to
\begin{equation}
\delta_n=1-\sigma Q_n(1-Q_n\theta(Q_n))^{\sigma-1}
\end{equation}
For $\sigma=2$ we can find explicit formulas for these parameters: $p_c=\frac{1}{2}$, $\theta(p)=p^{-2}(2p-1)$ for $p\geq p_c$, $\delta_n=2Q_n-1$ and $p_c(1-\delta_n)=1-Q_n$.  However, all the information needed for our purposes is contained in the asymptotic relation \eqref{QndeltanRelation}.

\subsection{The outlet weight process}
\label{RegularTreeCaseSubsection}

The representation \eqref{RecursiveLawofQn} simplifies dramatically when $\underlyinggraph$ is a regular tree.  Then the connected component of $\underlyinggraph\backslash \tilde{\calC}_n$ containing $\overline{e}_n$ is isomorphic to $\underlyinggraph$, with $\overline{e}_n$ corresponding to the root.  Therefore the dependence of $Q_{n+1}$ on $e_n$ and $\tilde{\calC}_n$ is eliminated and we have the following result.
\begin{coro}
On the regular tree, $(Q_n)_{n=1}^\infty$ is a time-homogeneous Markov chain with
\begin{equation}
\label{Q1cdf}
\P(Q_1<q)=\theta(q)
\end{equation}
and
\begin{equation}
\label{QnextFromQn}
\condP{Q_{n+1}<q'}{Q_n=q}=\frac{\theta(q')}{\theta(q)}
\end{equation}
for $p_c<q'<q$.
\end{coro}

Equations \eqref{Q1cdf} and \eqref{QnextFromQn} say that, conditional on $Q_n$, $Q_{n+1}$ is chosen from the same distribution, conditioned to be smaller.  In terms of $(W_k)_{k=0}^\infty$, \eqref{QnextFromQn} describes the jumps of $W_k$ when they occur, and indeed the transition mechanism \eqref{QnextFromQn} is implicit in \cite{AGdHS2008}.  

Since $\theta$ is a continuous function, it is simpler to consider $\theta(Q_n)$: $\theta(Q_1)$ is uniform on $[0,1]$ and
\begin{equation}
\condP{\theta(Q_{n+1})<u'}{\theta(Q_n)=u}=\frac{u'}{u}
\end{equation}
for $0<u'<u$.  But this is equivalent to multiplying $\theta(Q_n)$ by an independent Uniform$[0,1]$ variable.  Noting further that the logarithm of a Uniform$[0,1]$ variable is exponential of mean 1, we have proved the following proposition.

\begin{prop}
Let $U_i$, $i\in\N$, be independent Uniform$[0,1]$ random variables.  Then, as processes,
\begin{equation}
\Bigl(\theta(Q_n)\Bigr)_{n=1}^\infty \overset{d}{=} \left(\prod_{i=1}^n U_i\right)_{n=1}^\infty
\end{equation}
Equivalently, with $E_i=\log U_i^{-1}$ independent exponentials of mean 1,
\begin{equation}
\label{logthetaQnRWrep}
\log\!\left(\theta(Q_n)^{-1}\right)\overset{d}{=}\sum_{i=1}^n E_i
\end{equation}
jointly for all $n$.
\end{prop}

\begin{coro}
\label{QndeltanetanCorollary}
The triple 
\begin{equation}
\vec{Z}'_n=\left(\log\!\left(\theta(Q_n)^{-1}\right)\! ,\log\left((Q_n-p_c)^{-1}\right)\! ,\log\delta_n^{-1}\right)
\end{equation}
satisfies the conclusions of Theorems \ref{SLLNtheorem} and \ref{CLTtheorem}, and each component of $\frac{1}{n}\vec{Z}'_n$ satisfies a large deviation principle with rate $n$ and rate function
\begin{equation}
\phi(u)=u-\log u -1.
\end{equation}
\end{coro}
\begin{proof}
The conclusions about $\log\left(\theta(Q_n)^{-1}\right)$ follow from the representation \eqref{logthetaQnRWrep} in terms of a sum of independent variables; the rate function $\phi$ is given by Cram\'{e}r's theorem.  The other results then follow from the asymptotic relation \eqref{QndeltanRelation}.
\end{proof}

\section{Law of large numbers and central limit theorem}
\label{LLNandCLTproofSection}

\subsection{Tail bounds for pond statistics}
\label{PondTailBoundsSection}

Theorem \ref{InvasionStructureTheorem} expressed $L_n, R_n$ and $V_n$ as random variables whose parameters are given in terms of $Q_n$.  Their fluctuations are therefore a combination of fluctuations arising from $Q_n$, and additional randomness.  The following proposition gives bounds on the additional randomness.

Recall that $\delta_n$ is a certain function of $Q_n$ with $\delta_n\sim\sigma(Q_n-p_c)$: see Theorem \ref{InvasionStructureTheorem}.

\begin{prop}
\label{LnRnVnBoundProp}
There exist positive constants $C,c,s_0,\gamma_L,\gamma_R,\gamma_V$ such that $L_n$, $R_n$ and $V_n$ satisfy the conditional bounds
\begin{align}
\label{LTailUpperBounds}
\condP{\delta_n L_n\geq S}{\delta}
&\leq Ce^{-cS}
&\condP{\delta_n L_n\leq s}{\delta}
&\leq Cs\\
\label{RTailUpperBounds}
\condP{\delta_n R_n\geq S}{\delta}
&\leq Ce^{-cS}
&\condP{\delta_n R_n\leq s}{\delta}
&\leq Cs\\
\label{VTailUpperBounds}
\condP{\delta_n^2 V_n\geq S}{\delta}
&\leq Ce^{-cS}
&\condP{\delta_n^2 V_n\leq s}{\delta}
&\leq C\sqrt{s}
\end{align}
for all $n$ and all $S,s>0$; and
\begin{align}
\label{LTailLowerBound}
\condP{\delta_n L_n\leq s}{\delta}
&\geq cs
&&\text{on $\set{\delta_n\leq \gamma_L s}$}\\
\label{RTailLowerBound}
\condP{\delta_n R_n\leq s}{\delta}
&\geq cs
&&\text{on $\set{\delta_n\leq \gamma_R s}$}\\
\label{VTailLowerBound}
\condP{\delta_n^2 V_n\leq s}{\delta}
&\geq c\sqrt{s}
&&\text{on $\set{\delta_n^2\leq \gamma_V s}$}
\end{align}
for $s\leq s_0$.
\end{prop}

The proofs of \eqref{LTailUpperBounds}--\eqref{VTailLowerBound}, which involve random walk and branching process estimates, are deferred to section \ref{PondTailBoundsProofSection}.

\subsection{A uniform convergence lemma}\label{UniformConvergenceSection}

Because Theorem \ref{CLTtheorem} involves weak convergence of several processes to the same joint limit, it will be convenient to use Skorohod's representation theorem and almost sure convergence.  The following uniform convergence result will be used to extend convergence from one set of coupled random variables ($\delta_{n,N}$) to another ($X_{n,N}$): see section \ref{SLLNCLTproof}.

\begin{lemma}
\label{XnNUniformLimitLemma}
Suppose $\set{X_{n,N}}_{n,N\in\N}$ and $\set{\delta_{n,N}}_{n,N\in\N}$ are positive random variables such that $\delta_{n,N}$ is decreasing in $n$ for each fixed $N$, and for positive constants $a$, $\beta$ and $C$,
\begin{align}
\label{XnNUpperTail}
\P(\delta_{n,N}^a X_{n,N} > S) &\leq CS^{-\beta}\\
\label{XnNLowerTail}
\P(\delta_{n,N}^a X_{n,N} < s) &\leq Cs^\beta
\end{align}
for all $S$ and $s$.  Define
\begin{equation}
\hat{X}_{n,N}=\sum_{i=1}^n X_{i,N}.
\end{equation}
Then for any $T>0$ and $\alpha>0$, w.p. 1,
\begin{equation}
\label{XnNUniformLimit}
\lim_{N\goesto\infty}\max_{1\leq n\leq N T} \frac{\log (\delta_{n,N}^a X_{n,N})}{N^\alpha}=\lim_{N\goesto\infty}\max_{1\leq n\leq N T} \frac{\log (\delta_{n,N}^a \hat{X}_{n,N})}{N^\alpha}=0.
\end{equation}
\end{lemma}

\begin{proof}
Let $\epsilon>0$ be given.  For a fixed $N$, \eqref{XnNUpperTail} implies
\begin{align}
\P\left(\max_{1\leq n\leq NT} \frac{\log (\delta_{n,N}^a\hat{X}_{n,N})}{N^\alpha} >\epsilon\right)
&\leq
\sum_{1\leq n\leq NT}\P\left(\delta_{n,N}^a \hat{X}_{n,N} > e^{N^\alpha \epsilon}\right)\notag\\
&\leq
\sum_{1\leq n\leq NT}\sum_{i=1}^n \P\left(\delta_{n,N}^a X_{i,N}>\frac{e^{N^\alpha \epsilon}}{n}\right)\notag\\
&\leq
\sum_{1\leq n\leq NT}\sum_{i=1}^n \P\left(\delta_{i,N}^a X_{i,N}>\frac{e^{N^\alpha \epsilon}}{n}\right)\notag\\
&\leq
\sum_{1\leq n\leq NT}\sum_{i=1}^n C n^\beta e^{-\beta N^\alpha \epsilon}\notag\\
&\leq
(NT)^{2+\beta}Ce^{-\beta N^\alpha \epsilon}
\end{align}
where we used $\delta_{i,N}\geq\delta_{n,N}$ in the third inequality.  But then since $\sum_{N=1}^\infty (NT)^{2+\beta}Ce^{-\beta N^\alpha \epsilon} <\infty$, the Borel-Cantelli lemma implies
\begin{equation}
\label{hatXnNLimsupBound}
\limsup_{N\goesto\infty} \max_{1\leq n\leq N T} \frac{\log (\delta_{n,N}^a \hat{X}_{n,N})}{N^\alpha} \leq \epsilon
\end{equation}
a.s.  Similarly, \eqref{XnNLowerTail} implies
\begin{align}
\P\left(\max_{1\leq n\leq NT} \frac{\log(\delta_{n,N}^a X_{n,N})}{N^\alpha} < -\epsilon\right)
&\leq 
\sum_{1\leq n\leq NT} \P\left(\delta_{n,N}^a X_{n,N} < e^{-N^\alpha \epsilon}\right)\notag\\
&\leq NTCe^{-\beta N^\alpha \epsilon}
\end{align}
so that 
\begin{equation}
\label{XnNLiminfBound}
\liminf_{N\goesto\infty} \max_{1\leq n\leq N T} \frac{\log (\delta_{n,N}^a X_{n,N})}{N^\alpha} \geq -\epsilon
\end{equation}
a.s.  Since $\epsilon>0$ was arbitrary and $X_{n,N}\leq \hat{X}_{n,N}$, \eqref{XnNUniformLimit} follows.
\end{proof}

\subsection{Proof of Theorems \ref{SLLNtheorem}--\ref{CLTtheorem}}
\label{SLLNCLTproof}

The conclusions about $Q_n$ are contained in Corollary \ref{QndeltanetanCorollary}.  The other conclusions will follow from Lemma \ref{XnNUniformLimitLemma}. From Corollary \ref{QndeltanetanCorollary}, we may apply Skorohod's representation theorem to produce realizations of the ponds for each $N\in\N$, coupled so that
\begin{equation}
\label{deltanNCoupling}
\left(\frac{\log(\delta_{\floor{Nt},N}^{-1})-Nt}{\sqrt{N}}\right)_{0\leq t\leq T} \goesto (B_t)_{0\leq t\leq T}
\end{equation}
a.s. as $N\goesto\infty$.  Then the relation
\begin{equation}
\frac{\frac{1}{a}\log X_{n,N}-Nt}{N^{1/2}}=\frac{\log\left(\delta_{\floor{Nt},N}^{-1}\right)-Nt}{N^{1/2}}+\frac{\log\left(\delta_{\floor{Nt},N}^a X_{\floor{Nt},N}\right)}{a N^{1/2}}
\end{equation}
shows that $\frac{1}{a}\log X_n$ will satisfy a central limit theorem as well.  The same holds for $\hat{X}_n$.  We will successively set
\begin{align}
X_{n,N}
&=
L_{n,N},&&\text{with $a=1$,}\\
X_{n,N}
&=
R_{n,N},&&\text{with $a=1$,}\\
X_{n,N}
&
=V_{n,N},&&\text{with $a=2$.}
\end{align}
The bounds \eqref{XnNUpperTail}--\eqref{XnNLowerTail} follow immediately from the bounds in Proposition \ref{LnRnVnBoundProp}.   This proves Theorem \ref{CLTtheorem} for $L_n$ and $V_n$.  For $R_n$, the quantity $\hat{R}$ is not the one that appears in Theorem \ref{CLTtheorem}, but the bound $R_n\leq R'_n\leq \hat{R}_n$ implies the result for $R'_n$ as well.

The lemma also implies the law of large numbers results \eqref{ZnSLLN}, by taking $T=1$ and using the same ponds for every $N$.

\section{Large deviations: proof of Theorem \ref{LDtheorem}}

In this section we present a proof of the large deviation results in Theorem \ref{LDtheorem}.  As in section \ref{LLNandCLTproofSection}, we prove a generic result using a variable $X_n$ and tail estimates.  Theorem \ref{LDtheorem} then follows immediately using Corollary \ref{QndeltanetanCorollary} and Proposition \ref{LnRnVnBoundProp}.

Note that Proposition \ref{LDProp} uses the full strength of the bounds in Proposition \ref{LnRnVnBoundProp}.

\begin{prop}
\label{LDProp}
Suppose that $\delta_n$ and $X_n$ are positive random variables such that, for positive constants $a,\beta,c,C,\gamma,s_0$, 
\begin{align}
\label{XnStrongerUpperTail}
\condP{\delta_n^a X_n > S}{\delta_n} &\leq Ce^{-cS^\beta}\\
\label{XnLowerTailUpperBound}
\condP{\delta_n^a X_n < s}{\delta_n} &\leq Cs^{1/a}\\
\intertext{for all $S$ and $s$, and}
\label{XnLowerTailLowerBound}
\condP{\delta_n^a X_n < s}{\delta_n} &\geq cs^{1/a}
\end{align}
on the event $\set{\delta_n^a<\gamma s}$, for $s\leq s_0$.  Suppose also that $\frac{1}{n}\log \delta_n^{-1}$ satisfies a large deviation principle with rate $n$ on $[0,\infty)$ with rate function $\phi$ such that $\phi(1)=0$, $\phi$ is continuous on $(0,\infty)$, and $\phi$ is decreasing on $(0,1]$ and increasing on $[1,\infty)$.  Then $\frac{1}{an}\log X_n$ satisfies a large deviation principle with rate $n$ on $[0,\infty)$ with rate function 
\begin{equation}
\label{ModifiedRateFunction}
\psi(u)=
\inf_{v\geq u} \bigl(\phi(v)+v-u\bigr)
\end{equation}
\end{prop}

\begin{proof}
It is easy to check that $\psi$ is continuous, decreasing on $[0,1]$ and increasing on $[1,\infty)$, $\psi(1)=0$, and $\psi(u)=\phi(u)$ for $u\geq 1$.  So it suffices to show that
\begin{equation}
\label{XnLDBoundAbove}
\lim_{n\goesto\infty} \frac{1}{n}\log \P\left(\frac{1}{an}\log X_n>u\right) =-\inf_{v>u}\phi(v)
\end{equation}
for $u>0$ and
\begin{equation}
\label{XnLDBoundBelow}
\lim_{n\goesto\infty} \frac{1}{n}\log \P\left(\frac{1}{an}\log X_n<u\right) =-\psi(u)
\end{equation}
for $0<u<1$.  For \eqref{XnLDBoundAbove}, let $\epsilon>0$.  Then
\begin{align}
\P\left(\frac{1}{an}\log X_n>u\right)
&\leq
\P\left(\frac{1}{n}\log\delta_n^{-1}>u-\epsilon\right)
\notag\\
&\quad
+\P\left(\frac{1}{n}\log\delta_n^{-1}\leq u-\epsilon, \frac{1}{an}\log X_n>u\right)\notag\\
&\leq
\P\left(\frac{1}{n}\log\delta_n^{-1}>u-\epsilon\right)
+\P\left(\frac{1}{an}\log(\delta_n^a X_n)>\epsilon\right)\notag\\
\label{LDUpperTailBoundAbove}
&\leq
\P\left(\frac{1}{n}\log\delta_n^{-1}>u-\epsilon\right)+Ce^{-ce^{\beta an\epsilon}}
\end{align}
where we used \eqref{XnStrongerUpperTail} with $S=e^{an\epsilon}$.  The last term in \eqref{LDUpperTailBoundAbove} is super-exponentially small, so \eqref{LDUpperTailBoundAbove} and the large deviation principle for $\frac{1}{n}\log\delta_n^{-1}$ imply
\begin{equation}
\limsup_{n\goesto\infty}\frac{1}{n}\log\P\left(\frac{1}{an}\log X_n>u\right)\leq -\inf_{v>u-\epsilon}\phi(v).
\end{equation}
On the other hand,
\begin{align}
\P\left(\frac{1}{an}\log X_n>u\right)
&\geq
\P\left(\frac{1}{n}\log\delta_n^{-1}>u+\epsilon, \frac{1}{an}\log (\delta_n^a X_n) > -\epsilon\right)\notag\\
&\geq \P\left(\frac{1}{n}\log\delta_n^{-1}>u+\epsilon\right)(1-Ce^{-n\epsilon})
\end{align}
using \eqref{XnLowerTailUpperBound} with $s=e^{-an\epsilon}$.  So
\begin{equation}
\liminf_{n\goesto\infty}\frac{1}{n}\log\P\left(\frac{1}{an}\log X_n>u\right)\geq -\inf_{v>u+\epsilon} \phi(v).
\end{equation}
Since $\phi$ is continuous and $\epsilon>0$ was arbitrary, this proves \eqref{XnLDBoundAbove}.

For \eqref{XnLDBoundBelow}, let $u\in(0,1)$ be given and choose $v\in(u,1)$, $\epsilon\in(0,u)$.  Then for $n$ sufficiently large we have
\begin{align}
\P\left(\frac{1}{an}\log X_n <u\right)
&\geq
\P\left(v-\epsilon<\frac{1}{n}\log\delta_n^{-1} <v, \frac{1}{an}\log (\delta_n^a X_n)<u-v\right)\notag\\
&\geq
\P\left(v-\epsilon<\frac{1}{n}\log\delta_n^{-1}<v\right)ce^{-n(v-u)}
\end{align}
Here we used \eqref{XnLowerTailLowerBound} with $s=e^{-an(v-u)}$.  Note that if $n$ is large enough then $s\leq s_0$ and the condition $\delta_n^a<\gamma s$ follows from $v-\epsilon<\frac{1}{n}\log\delta_n^{-1}$.  Therefore, since $\phi$ is decreasing on $(0,1]$,
\begin{align}
\liminf_{n\goesto\infty} \frac{1}{n}\log\P\left(\frac{1}{an}\log X_n <u\right)
&\geq
-\left(\inf_{v-\epsilon<w<v}\phi(w)\right)-(v-u)\notag\\
\label{LDLowerTailBoundBelow}
&=
-\phi(v)-a(v-u)
\end{align}
\eqref{LDLowerTailBoundBelow} was proved for $u<v<1$.  However, since $\phi$ is continuous and the function $-\phi(v)-av$ is decreasing in $v$ for $v\geq 1$, \eqref{LDLowerTailBoundBelow} holds for all $v\geq u$.  So take the supremum over $v\geq u$ to obtain
\begin{equation}
\liminf_{n\goesto\infty} \frac{1}{n}\log\P\left(\frac{1}{an}\log X_n <u\right) \geq -\psi(u).
\end{equation}
Finally 
\begin{align}
&\P\left(\frac{1}{an}\log X_n <u\right)\notag\\
&\quad\leq
\P\left(\frac{1}{n}\log\delta_n^{-1}\leq u\right)
+\E\left(1\left(\frac{1}{n}\log\delta_n^{-1}>u\right)\condP{\frac{1}{an}\log X_n<u}{\delta_n}\right)\notag\\
&\quad=
\P\left(\frac{1}{n}\log\delta_n^{-1}\leq u\right)\notag\\
&\qquad
+\E\left(1\left(\frac{1}{n}\log\delta_n^{-1}>u\right) \condP{\frac{1}{an}\log (\delta_n^a X_n)<u-\frac{1}{n}\log\delta_n^{-1}}{\delta_n}\right)\notag\\
\label{LDLowerTailUpperBound}
&\quad\leq
\P\left(\frac{1}{n}\log\delta_n^{-1}\leq u\right) +\E\left(1\left(\frac{1}{n}\log\delta_n^{-1}>u\right)Ce^{n\left(u-\frac{1}{n}\log\delta_n^{-1}\right)}\right)
\end{align}
(using \eqref{XnLowerTailUpperBound} with $s=e^{an\left(u-\frac{1}{n}\log\delta_n^{-1}\right)}$).  Apply Varadhan's lemma (see, e.g., \cite[p. 32]{dH2000}) to the second term of \eqref{LDLowerTailUpperBound}:
\begin{equation}
\begin{split}
\lim_{n\goesto\infty}\frac{1}{n}\log\E\left(1\left(\frac{1}{n}\log\delta_n^{-1}>u\right)Ce^{n\left(u-\frac{1}{n}\log\delta_n^{-1}\right)}\right)\\
\qquad=\sup_{v>u}(u-v-\phi(v))=-\psi(u)
\end{split}
\end{equation}
Therefore
\begin{equation}
\limsup_{n\goesto\infty}\frac{1}{n}\log\P\left(\frac{1}{an}\log X_n <u\right) \leq \max\set{-\phi(u),-\psi(u)}  = -\psi(u)
\end{equation}
which completes the proof.
\end{proof}

\section{Tail asymptotics}\label{TailAsymptoticsSection}

In this section we prove the fixed-pond asymptotics from Theorem \ref{QnLnAsymp}.

\begin{proof}[Proof of \eqref{QnTail}]
Recall from \eqref{logthetaQnRWrep} that $\log\!\left(\theta(Q_n)^{-1}\right)$ has the same distribution as a sum of $n$ Exponential variables of mean 1, i.e., a Gamma variable with parameters $n,1$.  So
\begin{align}
\P(\theta(Q_n)<\epsilon)
&=
\P\left(\log\!\left(\theta(Q_n)^{-1}\right)>\log \epsilon^{-1}\right)\notag\\
&=
\int_{\log\epsilon^{-1}}^\infty \frac{x^{n-1}}{(n-1)!} e^{-x}\dx
\end{align}
Make the substitution $x=(1+u)\log\epsilon^{-1}$:
\begin{align}
\P(\theta(Q_n)<\epsilon)
&=
\frac{\epsilon\left(\log\epsilon^{-1}\right)^n}{(n-1)!} \int_0^\infty (1+u)^{n-1} e^{-u\log\epsilon^{-1}}
\end{align}
Then Watson's lemma (see for instance (2.13) of \cite{MurrayAsymp1984}) implies that 
\begin{equation}
\label{thetaQnAsymptotics}
\P(\theta(Q_n)<\epsilon)\sim\frac{\epsilon\left(\log\epsilon^{-1}\right)^{n-1}}{(n-1)!}
\end{equation}
and so \eqref{QndeltanRelation} gives 
\begin{equation}
\P(Q_n<p_c(1+\epsilon))\sim\frac{2\sigma\epsilon\left(\log\epsilon^{-1}\right)^{n-1}}{(\sigma-1)(n-1)!}
\end{equation}
\end{proof}

Combining \eqref{thetaQnAsymptotics} with \eqref{QndeltanRelation} implies at once that 
\begin{equation}
\label{deltanTail}
\P\left(\delta_n<\epsilon\right)\asymp \epsilon(\log \epsilon^{-1})^{n-1}
\end{equation}
We use \eqref{deltanTail} to prove \eqref{RnTail}--\eqref{VnTail} using the following lemma.

\begin{lemma}
\label{XnTailLemma}
Let $\delta_n$ be a random variable satisfying \eqref{deltanTail}.  Suppose $a,\beta$ are positive constants such that $a\beta>1$, and $X_n$ is any positive random variable satisfying
\begin{equation}
\label{XnAsympUpperTail}
\condP{\delta_n^a X_n>S}{\delta}\leq C S^{-\beta}
\end{equation}
for all $S,n>0$, and
\begin{equation}
\label{XnWeakLowerTail}
\condP{\delta_n^a X_n>s_0}{\delta}\geq p_0
\end{equation}
for some $s_0,p_0>0$.  Write $\hat{X}_n=\sum_{i=1}^n X_i$.  Then
\begin{equation}
\label{XnAsymp}
\P(X_n>k)\asymp\P\left(\hat{X}_n>k\right)\asymp\frac{(\log k)^{n-1}}{k^{1/a}}
\end{equation}
as $k\goesto\infty$.
\end{lemma}

\begin{proof}
From \eqref{XnWeakLowerTail} and \eqref{deltanTail},
\begin{align}
\P(X_n>k)
&\geq 
\condP{\delta_n^a X_n>s_0 \Big.}{\delta_n<s_0 k^{-1/a}}\P(\delta_n<s_0 k^{-1/a})\notag\\
&\geq
\newnumberedconstant\frac{(\log k)^{n-1}}{k^{1/a}}
\end{align}
For the lower bound we condition on the Gamma random variable $\log\theta(Q_n)^{-1}$.  From \eqref{QndeltanRelation} we have $\delta_n\geq \namednumberedconstant{deltathetaconstant}\theta(Q_n)$ for some constant $\previousconstant{deltathetaconstant}>0$, so that
\begin{align}
\P(X_n>k)
&=
\E\left(\condP{\delta_n^a X_n>k\delta_n^a}{\delta_n}\right)\notag\\
&\leq
\E\left(1\wedge C(k\delta_n^a)^{-\beta}\right)\notag\\
&\leq
\newnumberedConstant\E\left(1\wedge \left(k\theta(Q_n)^a\right)^{-\beta}\right)\notag\\ \label{XnTailAsympUnsubbedIntegral}
&=
\newnumberedConstant\int_0^\infty \left(1\wedge \left(ke^{-ax}\right)^{-\beta}\right)x^{n-1}e^{-x}\dx
\end{align}
Make the substitution $y=ke^{-ax}$ to obtain
\begin{align}
&\P(X_n>k)
\leq
\newnumberedConstant\int_0^k \left(1\wedge y^{-\beta}\right)\left(\log k-\log y\right)^{n-1} \left(\frac{y}{k}\right)^{1/a}\frac{dy}{y}\notag\\
\label{XnTailAsympIntegral}
&\quad\leq
\frac{\lastConstant(\log k)^{n-1}}{k^{1/a}}\int_0^\infty \left(1\wedge y^{-\beta}\right)\left(1+\frac{\abs{\log y}}{\log k}\right)^{n-1} \frac{dy}{y^{1-1/a}}
\end{align}
The last integral in \eqref{XnTailAsympIntegral} is bounded as $k\goesto\infty$: the singularity as $y\goesto 0^+$ is integrable since $1-1/a>0$, and the singularity as $y\goesto\infty$ is integrable since the exponent in $y^{-1-\beta+1/a}$ has $-1-\beta+1/a<-1$ since $a\beta>1$.

To extend \eqref{XnAsymp} to $\hat{X}_n$, assume inductively that $\P(\hat{X}_n>k)\asymp (\log k)^{n-1}/k^{1/a}$.  (The case $n=1$ is already proved since $\hat{X}_1=X_1$.)  The bound $\P(\hat{X}_{n+1}>k)\geq \P(X_{n+1}>k)$ is immediate, and we can estimate
\begin{align}
\P\left(\hat{X}_{n+1}>k\right)\leq \P\left(X_{n+1} >k-k'\right) +\P\left(\hat{X}_n >k'\right)
\end{align}
where we set $k'=\floor{k/(\log k)^{a/2}}$.  Then $k-k'\sim k$ and $\log (k-k')\sim\log k' \sim \log k$, so that
\begin{equation}
\P\left(X_{n+1} >k-k'\right)\asymp \frac{(\log k)^n}{k^{1/a}}
\end{equation}
while
\begin{equation}
\P\left(\hat{X}_n >k'\right)\asymp \frac{(\log k')^{n-1}}{\left(k'\right)^{1/a}}\asymp \frac{(\log k)^{n-1/2}}{k^{1/a}}
\end{equation}
which is of lower order.  This completes the induction.
\end{proof}

\begin{proof}[Proof of \eqref{RnTail}--\eqref{VnTail}]
These relations follow immediately from \eqref{deltanTail} and Lemma \ref{XnTailLemma}; the bounds \eqref{XnAsympUpperTail}--\eqref{XnWeakLowerTail} are immediate consequences of Proposition \ref{LnRnVnBoundProp}.  As in section \ref{SLLNCLTproof}, the asymptotics for $R'_n$ follow from the asymptotics for $\hat{R}_n$ and the bound $R_n\leq R'_n\leq \hat{R}_n$.
\end{proof}

\begin{proof}[Proof of \eqref{LnTail}]
For $L_n$, we can use the exact formula $\condP{L_n>k}{\delta_n}=(1-\delta_n)^k$.  Write $\delta_n=g(\theta(Q_n))$, where $g(p)$ is a certain continuous and increasing function.  By \eqref{QndeltanRelation}, $g(p)\sim \frac{\sigma-1}{2\sigma} p$ as $p\goesto 0^+$; as above we will use the bound $g(p)\geq cp$ for some constant $c>0$.  Proceeding as in \eqref{XnTailAsympUnsubbedIntegral} and \eqref{XnTailAsympIntegral} we obtain the exact formula
\begin{align}
&\P(L_n>k)
=
\frac{1}{k(n-1)!}\int_0^k \bigl(1-g(y/k)\bigr)^k (\log k -\log y)^{n-1} \dy\notag\\
\label{LnAbovekFormula}
&\quad=
\frac{(\log k)^{n-1}}{k(n-1)!} \int_0^\infty \mathbbm{1}(y\leq k)\bigl(1-g(y/k)\bigr)^k \left(1-\frac{\log y}{\log k}\right)^{n-1} dy
\end{align}
But the integral in \eqref{LnAbovekFormula} converges to $\int_0^\infty e^{-\frac{\sigma-1}{2\sigma}y}\dy=\frac{2\sigma}{\sigma-1}$ as $k\goesto\infty$: pointwise convergence follows from $g(p)\sim \frac{\sigma-1}{2\sigma} p$, and we can uniformly bound the integrand using 
\begin{equation}
\bigl(1-g(y/k)\bigr)^k \leq e^{-kg(y/k)}\leq e^{-cy}
\end{equation}
Lastly, a simple modification of the argument for $\hat{X}_n$ extends \eqref{LnTail} to $\hat{L}_n$.
\end{proof}

\section{Pond bounds: proof of Proposition \ref{LnRnVnBoundProp}}
\label{PondTailBoundsProofSection}

In this section we prove the tail bounds \eqref{LTailUpperBounds}--\eqref{VTailLowerBound}.  Since the laws of $L_n$, $R_n$ and $V_n$ do not depend on $n$ except through the value of $\delta_n$, we will omit the subscript in this section.  For convenient reference we recall the structure of the bounds:
\begin{align}
\label{XStrongerUpperTail}
\condP{\delta^a X > S }{\delta} &\leq Ce^{-cS},\\
\label{XLowerTailUpperBound}
\condP{\delta^a X < s }{ \delta} &\leq Cs\\
\intertext{for all $S$ and $s$, and}
\label{XLowerTailLowerBound}
\condP{\delta^a X < s }{\delta} &\geq cs^{1/a}
\end{align}
on the event $\set{\delta^a<\gamma s}$, for $s\leq s_0$.  We have $a=1$ for $X=L$ and $X=R$, and $a=2$ for $X=V$.

In \eqref{XLowerTailLowerBound} it is necessary to assume $\delta^a<\gamma s$.  This is due only to a discretization effect: for any $\N$-valued random variable $X$, we have $\condP{\delta^a X<s}{\delta}=0$ whenever $\delta^a<s$.  Indeed, the bounds \eqref{LTailLowerBound}--\eqref{VTailLowerBound} can be proved with $\gamma=1$, although it is not necessary for our purposes.

Note that, by proper choice of $C$ and $s_0$, we can assume that $S$ is large in \eqref{XStrongerUpperTail} and $s$ is small in \eqref{XLowerTailUpperBound} and \eqref{XLowerTailLowerBound}.  Since we only consider $\N$-valued random variables $X$, we can assume $\delta$ is small in \eqref{XLowerTailUpperBound}, say $\delta<1/2$ (otherwise take $s<(1/2)^a$ without loss of generality). Moreover, Theorem \ref{InvasionStructureTheorem} shows that $L$, $R$ and $V$ are all stochastically decreasing in $\delta$.  Consequently it suffices to prove \eqref{XStrongerUpperTail} for $\delta$ small, say $\delta<1/2$.  Finally the constraint $\delta^a<\gamma s_0$ makes $\delta$ small in \eqref{XLowerTailLowerBound} also. 

We note for subsequent use the inequalities
\begin{equation}
(1-x)^y\leq e^{-xy}
\end{equation}
for $x\in(0,1),y>0$ and
\begin{equation}
\label{1minus1minusxtotheybound}
1-(1-x)^y\leq 1-e^{-2xy}\leq 2xy
\end{equation}
for $x\in(0,1/2),y>0$, which follow from $\log(1-x)\leq -x$ for $x\in(0,1)$ and $\log(1-x)\geq -2x$ for $x\in(0,1/2)$.

\subsection{The backbone length \texorpdfstring{$L$}{L}: proof of \texorpdfstring{\eqref{LTailUpperBounds}}{(\ref{LTailUpperBounds})} and \texorpdfstring{\eqref{LTailLowerBound}}{(\ref{LTailLowerBound})}}

From Theorem \ref{InvasionStructureTheorem}, $L$ is a geometric random variable with parameter $\delta$.  So
\begin{align}
\condP{L>S/\delta}{ \delta}
&=
(1-\delta)^{\floor{S/\delta}}\notag\\
\label{LStrongerUpperTail}
&\leq
e^{-\delta\floor{S/\delta}}
\leq e^{-S+\delta}\leq e^{-S+1}
\end{align}
since $\delta\leq 1$, proving \eqref{XStrongerUpperTail}.  For \eqref{XLowerTailUpperBound} and \eqref{XLowerTailLowerBound}, we have
\begin{equation}
\label{LLowerTailFormula}
\condP{L < s/\delta}{ \delta}
=
1-(1-\delta)^{\ceil{s/\delta}-1}.
\end{equation}
For $\delta\leq\frac{1}{2}$ we can use \eqref{1minus1minusxtotheybound} to get
\begin{align}
\condP{L < s/\delta}{\delta}
&\leq
2\delta(\ceil{s/\delta}-1)\leq 2s
\end{align}
which proves \eqref{XLowerTailUpperBound}.  For \eqref{XLowerTailLowerBound}, take $\gamma_L=1/2$.  Then on the event $\set{\delta<\gamma_L s}$ we have $\ceil{s/\delta}\geq 3$ so that expanding \eqref{LLowerTailFormula} as a binomial series gives
\begin{align}
\condP{L < s/\delta}{ \delta}
&\geq
\left(\ceil{s/\delta}-1\right)\delta-\frac{1}{2}\left(\ceil{s/\delta}-1\right)\left(\ceil{s/\delta}-2\right)\delta^2\notag\\
&\geq
(s-\delta)-\frac{1}{2}s(s-\delta)\notag\\
\label{LLowerTailLowerBound}
&\geq
\frac{s}{2}\left(1-\frac{s}{2}\right)\geq \frac{s}{4}
\end{align}
for $s\leq 1=s_0$.  So \eqref{XLowerTailLowerBound} holds.

\subsection{The pond radius \texorpdfstring{$R$}{R}: proof of \texorpdfstring{\eqref{RTailUpperBounds}}{(\ref{RTailUpperBounds})} and \texorpdfstring{\eqref{RTailLowerBound}}{(\ref{RTailLowerBound})}}

Conditional on $\delta$ and $L$, $R$ is the maximum height of a percolation cluster with parameter $p_c(1-\delta)$ started from a path of length $L$.  We have $R\geq L$ so \eqref{XLowerTailUpperBound} follows immediately from the corresponding bound for $L$.  $R$ is stochastically dominated by 
\begin{equation}
L+\max_{1\leq i\leq L} \tilde{R}_i
\end{equation}
where $\tilde{R}_i$ is the maximum height of a branching process with offspring distribution Binomial($\sigma,p_c(1-\delta)$) started from a single vertex, independently for each $i$.  Define
\begin{equation}
a_k=\condP{\tilde{R}_i>k}{\delta}
\end{equation}
for $k>0$.  Thus $a_k$ is the probability that the branching process survives to generation $k+1$.  By comparison with a critical branching process, 
\begin{equation}
\label{CriticalBranchingProcessBound}
a_k\leq \frac{\namednumberedConstant{CritBPConst}}{k}
\end{equation}
for some constant $\previousConstant{CritBPConst}$.  On the other hand, $a_k$ satisfies
\begin{equation}
a_{k+1}=1-f(1-a_k)
\end{equation}
where $f(z)$ is the generating function for the offspring distribution of the branching process.  (This is a reformulation of the well-known recursion for the extinction probability.)  In particular, since $f'(z)\leq f'(1)=\sigma p_c(1-\delta)=1-\delta$,
\begin{equation}
\label{SubcriticalBranchingProcessRecursion}
a_{k+1}\leq f'(1)a_k=a_k(1-\delta).
\end{equation}
Combining \eqref{CriticalBranchingProcessBound} with \eqref{SubcriticalBranchingProcessRecursion},
\begin{equation}
a_{k+j}\leq a_k(1-\delta)^j\leq \frac{\previousConstant{CritBPConst}}{k}e^{-\delta j}
\end{equation}
and taking $k=\ceil{S/2\delta}\geq S/2\delta$, $j=\floor{S/\delta}-\ceil{S/2\delta}\geq S/2\delta-2$,
\begin{align}
\condP{\tilde{R}_i>\frac{S}{\delta}}{\delta}
&=
a_{\floor{S/\delta}}
\leq
\frac{2\previousConstant{CritBPConst}\delta}{S}e^{\delta(S/2\delta-2)}\notag\\
&\leq
\frac{\newnumberedConstant \delta e^{-\newnumberedconstant S}}{S}
\end{align}
Using this estimate we can compute
\begin{align}
&\condP{R>\frac{S}{\delta}}{\delta}\notag\\
&\qquad\leq
\condP{L>\frac{S}{2\delta}}{\delta}+\condP{L\leq \frac{S}{2\delta},\tilde{R}_i>\frac{S}{2\delta}\text{ for some $i\leq S/2\delta$}}{\delta}\notag\\
&\qquad\leq
\newnumberedConstant e^{-\newnumberedconstant S}+\left(\frac{S}{2\delta}\right) \frac{\newnumberedConstant \delta e^{-\newnumberedconstant S}}{S} 
\leq 
\newnumberedConstant e^{-\newnumberedconstant S}
\end{align}
Similarly
\begin{align}
\condP{R<\frac{s}{\delta}}{\delta}
&\geq
\condP{L<\frac{s}{2\delta}, \tilde{R}_i<\frac{s}{2\delta}\text{ for all $i<\frac{s}{2\delta}$}}{\delta}\notag\\
&\geq
\newnumberedconstant s\left(1-\frac{\newnumberedConstant\delta e^{-\newnumberedconstant s}}{s}\right)^{s/2\delta}\geq \newnumberedconstant s
\end{align}
provided $\delta$ is sufficiently small compared to $s$, i.e., provided $\gamma_R$ is small enough.

\subsection{The pond volume \texorpdfstring{$V$}{V}: proof of \texorpdfstring{\eqref{VTailUpperBounds}}{(\ref{VTailUpperBounds})} and \texorpdfstring{\eqref{VTailLowerBound}}{(\ref{VTailLowerBound})}}

From Theorem \ref{InvasionStructureTheorem}, conditional on $\delta$ and $L$, $V$ is the number of edges in a percolation cluster with parameter $p_c(1-\delta)$, started from a path of length $L$ and with no edges emerging from the top of the path.  We can express $V$ in terms of the return time of a random walk as follows.

Start with an edge configuration with $L$ backbone edges marked as occupied.  Mark as unexamined the $(\sigma-1)L$ edges adjacent to the backbone, not including the edges emerging from the top.  At each step, take an unexamined edge (if any remain) and either (1) with probability $1-p_c(1-\delta)$, mark it as vacant; or (2) with probability $p_c(1-\delta)$, mark it as occupied and mark its child edges as unexamined.  Let $N_k$ denote the number of unexamined edges after $k$ steps.  Then it is easy to see that $N_k$ is a random walk $N_k=N_0+\sum_{i=1}^k Y_i$ (at least until $N_k=0$) where $N_0=(\sigma-1)L$ and 
\begin{equation}
Y_i=\begin{cases}\sigma-1&\text{w.p. $p_c(1-\delta)$,}\\ -1&\text{w.p. $1-p_c(1-\delta)$.}\end{cases}
\end{equation}
Let $T=\inf\set{k : N_k=0}$.  ($T$ is finite a.s. since $\condE{Y_i}{\delta}=-\delta<0$ and $N_k$ can jump down only by 1.)  $T$ counts the total number of off-backbone edges examined, namely the number of non-backbone edges in the cluster and on its boundary, not including the edges from the top of the backbone.  Consequently 
\begin{equation}
T=[V-L]+[(\sigma-1)V+\sigma]-\sigma=\sigma V-L
\end{equation}
In order to apply random walk estimates we write $X_i=Y_i+\delta$, $Z_k=\sum_{i=1}^k X_i$, so that $\condE{X_i}{\delta}=0$; $\namednumberedconstant{XiVarianceLower}\leq\condE{X_i^2}{\delta}\leq \namednumberedConstant{XiVarianceUpper}$ for universal constants $\previousconstant{XiVarianceLower},\previousConstant{XiVarianceUpper}$; and $N_k=Z_k-k\delta+(\sigma-1)L$.  Note that 
\begin{align}
\set{V>V_0}
&=
\set{T >\sigma V_0-L}
\notag\\
&\subseteq
\set{Z_{\floor{\sigma V_0-L}}>\delta\floor{\sigma V_0-L}-(\sigma-1)L}
\end{align}
so using, for instance, Freedman's inequality \cite[Proposition 2.1]{Freedman1975} leads after some computation to
\newsavebox{\neverusedagain}
\sbox{\neverusedagain}{$\namednumberedconstant{filler}\namednumberedconstant{SLetabound}$}
\begin{align}
&\condP{V>S_0 L/\delta}{\delta,L}
\leq
\condP{Z_{\floor{L(\sigma S_0-\delta)/\delta}}>L(\sigma S_0-2\delta-\sigma+1)}{\delta,L}\notag\\
&\qquad\leq
\exp\left(-\frac{L^2(\sigma S_0-2\delta-\sigma+1)^2}{2(\sigma L(\sigma S_0-2\delta-\sigma+1)+\previousConstant{XiVarianceUpper} L(\sigma S_0-\delta)/\delta)}\right)\notag\\
&\qquad\leq
\exp\left(-\previousconstant{filler} S_0 L\delta(1-2/S_0)^2\right)
\leq \exp(-\previousconstant{SLetabound} S_0 L\delta)
\end{align}
if $S_0\geq 3$, say.  Then, setting $S_0=S/\delta	 L$,
\begin{align}
\condP{V>S/\delta^2}{\delta}
&\leq
\condP{L> S/3\delta}{\delta}+\condP{V>S_0 L/\delta, L\leq  S/3\delta}{\delta}\notag\\
&\leq
\newnumberedConstant e^{-\newnumberedconstant S}+\exp\left(-\previousconstant{SLetabound} (S/\delta L) L\delta \right)\notag\\
\label{VStrongerUpperTail}
&\leq
\newnumberedConstant e^{-\newnumberedconstant S}
\end{align}
which proves \eqref{XStrongerUpperTail}.

For \eqref{XLowerTailUpperBound}, apply Freedman's inequality again:
\begin{align}
\condP{V\leq s/\delta^2}{\delta,L}
&=
\condP{T\leq \sigma s/\delta^2-L}{\delta,L}\notag\\
&=
\condP{\min_{k\leq\sigma s/\delta^2-L}(Z_k-k\delta) \leq -(\sigma-1)L}{\delta,L}\notag\\
&\leq
\condP{\min_{k\leq\sigma s/\delta^2} Z_k\leq -(L-\delta(\sigma s/\delta^2))}{\delta,L}\notag\\
&\leq
\exp\left(-\frac{(L-\sigma s/\delta)^2}{2(\sigma(L-\sigma s/\delta)+\newnumberedConstant s/\delta^2)}\right)\label{VLowerTailFreedman}\\
&\leq
\exp\left(-c (\delta L-\sigma s)^2 /s\right)
\end{align}
(where in the denominator of \eqref{VLowerTailFreedman} we use $L\leq s/\delta^2$ since $V\geq L$).  So
\begin{align}
\label{VLowerTailUpperBoundSum}
\begin{split}
&\condP{V\leq s/\delta^2}{\delta}
\leq
\condP{L<2\sigma s/\delta}{\delta}\\
&\qquad+
\condE{\indicator{L\geq 2\sigma s/\delta} \exp\left(-c (\delta L-\sigma s)^2/s\right)}{\delta}
\end{split}
\end{align}
The first term in \eqref{VLowerTailUpperBoundSum} is at most $\newnumberedConstant s$ from \eqref{LTailUpperBounds} and will therefore be negligible compared to $\sqrt{s}$.  For the second term, note that for $L\geq \sigma s/\delta$,
\begin{equation}
\exp\left(-c(\delta L-\sigma s)^2/s\right)=\int_{\sigma s/\delta}^\infty \indicator{l>L} \frac{2c\delta(\delta l-\sigma s)}{s} e^{-c(\delta l-\sigma s)^2/s} dl
\end{equation}
and so, with the substitution $x\sqrt{s}=\delta l-\sigma s$,
\begin{align}
&\condE{\indicator{L\geq 2\sigma s/\delta} \exp\left(-c (\delta L-\sigma s)^2/s\right)}{\delta}\notag\\
&\qquad=
2c \int_{\sigma s/\delta}^\infty \condP{L<l}{\delta}\frac{\delta(\delta l-\sigma s)}{s} e^{-c(\delta l-\sigma s)^2/s} dl\notag\\
&\qquad=
2c \int_0^\infty \condP{L<\frac{x\sqrt{s}+\sigma s}{\delta}}{\delta} x e^{-cx^2} dx\notag\\
&\qquad\leq
\newnumberedConstant \int_0^\infty (x\sqrt{s}+\sigma s)xe^{-cx^2}dx
\leq
\newnumberedConstant \sqrt{s}
\end{align}
which proves \eqref{XLowerTailUpperBound}

Finally, for \eqref{XLowerTailLowerBound}, the Berry-Esseen inequality (see for instance \cite[Theorem 2.4.9]{Durrett2005}) implies
\begin{equation}
\abs{\condP{Z_k<-x\sqrt{k}\sqrt{\condE{X_i^2}{\delta}}}{\delta}-\Phi(-x)}\leq\frac{\namednumberedConstant{BEConst}}{\sqrt{k}}
\end{equation}
where $\Phi(x)=\P(G<x)$ for $G$ a standard Gaussian, and $\previousConstant{BEConst}$ is some absolute constant.  In particular, using $0<\previousconstant{XiVarianceLower}\leq \condE{X_i^2}{\delta}$,
\begin{equation}
\condP{Z_k<-(\sigma-1)\sqrt{k}}{\delta}\geq \namednumberedconstant{BEconsequenceconst} >0
\end{equation}
for $k\geq\namednumberedConstant{BEThreshold}$.  Choose $\gamma_V = 1\wedge \gamma_L^2\wedge (\previousConstant{BEThreshold}+1)^{-1}$.  Then we have $s/\delta^2\geq 1$ (so we may bound $\sigma s/\delta^2-\sqrt{s}/\delta\geq s/\delta^2$); $\sqrt{s}/\delta\leq \gamma_L$; and $\floor{s/\delta^2}\geq \previousConstant{BEThreshold}$, so that
\begin{align}
\condP{V<s/\delta^2}{\delta}
&=
\condP{T<\sigma s/\delta^2 - L}{\delta}\notag\\
&\geq
\condP{T<\sigma s/\delta^2 - \sqrt{s}/\delta,L<\sqrt{s}/\delta}{\delta}\notag\\
&\geq
\condP{T<s/\delta^2,L<\sqrt{s}/\delta}{\delta}\notag\\
&\geq
\condP{Z_{\floor{s/\delta^2}}< \delta\floor{s/\delta^2}-(\sigma-1)L,L<\sqrt{s}/\delta}{\delta}\notag\\
&\geq
\condP{Z_{\floor{s/\delta^2}}<-(\sigma-1)\sqrt{s/\delta^2}}{\delta}\condP{L<\sqrt{s}/\delta}{\delta}\notag\\
&\geq
\newnumberedconstant\sqrt{s}
\end{align}
proving \eqref{XLowerTailLowerBound}.

\setcounter{constantsubscript}{0} 

\section{Percolation with defects}\label{PercolationWithDefectsSection}

In this section we prove
\begin{equation}
\label{PercWithDefectsAsymptotics}
\P_{p_c}(o\overset{n}{\connectsto}\boundary B(k))\asymp k^{-2^{-n}}
\end{equation}
The case $n=0$ is a standard branching process result.  For $n>0$, proceed by induction.  Write $C(o)$ for the percolation cluster of the root $o$.  The lower bound follows from the following well-known estimate:
\begin{equation}
\label{ClusterSizeNAsymptotics}
\P_{p_c}(\abs{C(o)}>N)\asymp N^{-1/2}
\end{equation}
If $C(o)>N$ then there are at least $N$ vertices $v_1,\dotsc,v_N$ on the outer boundary of $C(o)$, any one of which may have a connection to $\partial B(k)$ with $n-1$ defects.  As a worst-case estimate we may assume that $v_1,\dotsc,v_N$ are still at distance $k$ from $\boundary B(k)$, so that by independence we have
\begin{align}
\P_{p_c}(o\overset{n}{\connectsto}\partial B(k))
&\geq
\P_{p_c}(\abs{C(o)}>N) \left(1-\left(1-\P_{p_c}\bigl(o\overset{n-1}{\connectsto}\boundary B(k)\bigr)\right)^N\right)\notag\\
&\geq
\frac{c_1}{\sqrt{N}} \left(1-(1-c_2 k^{-2^{-n+1}})^N\right)
\end{align}
for constants $c_1,c_2$.  If we set $N=k^{2^{-n+1}}$ then the second factor is of order 1, and the lower bound is proved.  For the upper bound, use a slightly stronger form of \eqref{ClusterSizeNAsymptotics} (see for instance \cite[p. 260]{GrimmettPerc1999}): 
\begin{equation}
\label{ClusterSizeNAsymptoticsStronger}
\P_{p_c}(\abs{C(o)}=N)\asymp (N+1)^{-3/2}
\end{equation}
Now if $C(o)=N$, with $N\leq k/2$, then there are at most $\sigma N$ vertices on the outer boundary of $C(o)$, one of which must have a connection with $n-1$ defects of length at least $k-N\geq k/2$.  So
\begin{align}
&\P_{p_c}(o\overset{n}{\connectsto}\boundary B(k))\notag\\
&\qquad\leq
\P_{p_c}(\abs{C(o)}>k/2)\notag\\
&\qquad\qquad
+\sum_{N=0}^{\floor{k/2}} \P_{p_c}(\abs{C(o)}=N) \left(1- \left(1-\P_{p_c}\bigl(0\overset{n-1}{\connectsto}\boundary B(k/2)\bigr)\right)^{\sigma N}\right)\notag\\
&\qquad\leq
\frac{c_3}{\sqrt{k}}+ \sum_{N=0}^\infty \frac{c_4}{(N+1)^{3/2}}\left(1- \bigr(1-c_5 k^{-2^{-n+1}}\bigr)^{\sigma N}\right)\notag\\
&\qquad\leq
\frac{c_3}{\sqrt{k}}+ \sum_{N<k^{2^{-n+1}}}\frac{c_6 k^{-2^{-n+1}}N}{(N+1)^{3/2}} +\sum_{N\geq k^{2^{-n+1}}} c_4 N^{-3/2}\notag\\
&\qquad\leq
c_3 k^{-2^{-1}}+c_7 \left(k^{2^{-n+1}}\right)^{-1/2}
\end{align}
which proves the result (the first term is an error term if $n\geq 2$ and combines with the second if $n=1$).

\bibliographystyle{plain}
\bibliography{../TeX/references}

\end{document}